\documentclass[11pt,leqno]{amsart}

\usepackage{graphicx}
\graphicspath{ {./images/} }

\usepackage[dvipsnames]{xcolor}
\definecolor{b}{HTML}{4472c4}
\definecolor{o}{HTML}{ED7D31}
\definecolor{g}{HTML}{70ad47}

\usepackage{fullpage}
\usepackage{amsmath,amsthm,amssymb,latexsym,color}
\usepackage[hidelinks,unicode,linktocpage=true]{hyperref}
\usepackage{lipsum}
\usepackage{enumitem,bbm}
\usepackage{hyperref}
\hypersetup{
    colorlinks=true,
    linkcolor=blue,
    filecolor=blue,
    citecolor=blue,
    urlcolor=cyan}

\usepackage{tikz}
\usepackage{subfigure}

\usepackage{multirow}

\newtheorem{theorem}{Theorem}[section]

\newtheorem{lemma}[theorem]{Lemma}
\newtheorem{corollary}[theorem]{Corollary}
\newtheorem{definition}[theorem]{Definition}

\newtheorem{fact}[theorem]{Fact}
\newtheorem{claim}[theorem]{Claim}

\newtheorem{problem}[theorem]{Problem}

\theoremstyle{remark}
\newtheorem{remark}[theorem]{Remark}

\theoremstyle{plain}

\newcommand{\R}{\mathbb{R}}

\def\dist{{\rm dist}}

\def\dist{{\rm dist}}

\newcommand{\seminorm}[1]{{\left\vert\kern-0.25ex\left\vert\kern-0.25ex\left\vert #1
    \right\vert\kern-0.25ex\right\vert\kern-0.25ex\right\vert}}

\begin{document}

\title[A universal threshold for geometric embeddings of trees]{A universal threshold for geometric embeddings of trees}

\author{Dylan J. Altschuler, Pandelis Dodos, Konstantin Tikhomirov and Konstantinos Tyros}

\address{Department of Mathematical Sciences, Carnegie Mellon University}
\email{daltschu@andrew.cmu.edu}

\address{Department of Mathematics, University of Athens, Panepistimiopolis 157 84, Athens, Greece}
\email{pdodos@math.uoa.gr}

\address{Department of Mathematical Sciences, Carnegie Mellon University}
\email{ktikhomi@andrew.cmu.edu}

\address{Department of Mathematics, University of Athens, Panepistimiopolis 157 84, Athens, Greece}
\email{ktyros@math.uoa.gr}



\begin{abstract}
A graph $G=(V,E)$ is {\it geometrically embeddable} into a normed space $X$ when there is a mapping $\zeta\colon V\to X$ such that
$\|\zeta(v)-\zeta(w)\|_X\leqslant 1$ if and only if $\{v,w\}\in E$, for all distinct $v,w\in V$. Our result is the following universal threshold for the embeddability of trees. Let $\Delta \geqslant 3$, and let $N$ be sufficiently large in terms of $\Delta$. Every $N$--vertex tree of maximal degree at most $\Delta$ is embeddable into any normed space of dimension at least $64\,\frac{\log N}{\log\log N}$, and complete trees are non-embeddable into any normed space of dimension less than $\frac{1}{2}\,\frac{\log N}{\log\log N}$. In striking contrast, spectral expanders and random graphs are known to be non-embeddable in sublogarithmic dimension. Our result is based on a randomized embedding whose analysis utilizes the recent breakthroughs on Bourgain's slicing problem.
\end{abstract}

\maketitle

\numberwithin{equation}{section}

\section{Introduction} \label{sec1}

Given a normed space $X$ and a graph $G=(V,E)$, a map $\zeta\colon V\to X$ is a {\it geometrical embedding} of $G$ into $X$ if
$$ E=\big\{\{v,w\}\subseteq V\colon \|\zeta(v)-\zeta(w)\|_X\leqslant 1\big\}. $$
Say that $G$ is non-embeddable into $X$ if no such embedding exists. Note that the mapping $\zeta$ exists if and only if $G$ is isomorphic
to an {\it intersection graph} in $X$ generated by parallel translates of the unit ball $B_X$.

The problem of identifying geometrically embeddable graphs has attracted significant attention, especially in the Euclidean setting.
It was proved in \cite{FM88} (see, also, \cite{RRS89} and \cite{RRS92}) that every $N$--vertex graph of maximal degree at most $\Delta$ admits a geometrical embedding into a Euclidean space of dimension $C\Delta^2\log N$, where $C>0$ is a universal constant. On the other hand, the Poincar\'e inequality implies that any spectral expander graph is non-embeddable into Euclidean space of sublogarithmic dimension \cite{AT25}. More generally, spectral expanders also do not embed into any spaces of sublogarithmic dimension satisfying a so-called nonlinear Poincar\'e inequality (see, for example, \cite{Na14, MN14, Esk22, ADTT24} for a discussion of such inequalities).

The smallest $m$ such that a graph $G$ can be embedded into an $m$--dimensional Euclidean space is called the {\it sphericity} of $G$, and denoted by ${\rm Sph}(G)$. The aforementioned results then imply that, suppressing dependence on $\Delta$, the sphericity of every $N$--vertex expander graph is of order $\Theta(\log N)$. We refer to \cite[Section~2]{MP93} and \cite[Section~3]{BL05} for a review of results on sphericity of various families of graphs. In the non-Euclidean setting, it was recently shown in~\cite{AT25} that with high probability, a random $\Delta$--regular graph on $N$ vertices cannot be geometrically embedded into {\it any} normed space of dimension $o(\log N)$. It remains an interesting open question whether this dimension bound is asymptotically optimal across all normed spaces.

The purpose of this note is to study geometrical embeddings of trees into arbitrary normed spaces. The setting of trees is fundamentally different from embedding random graphs despite the fact that sparse random graphs are locally tree-like, i.e., have few short cycles.
As an illustration, in \cite{FM86} it is shown that the complete rooted binary tree $T$ on $N$ vertices admits a geometrical embedding into $\ell_2$ of dimension at most $C\frac{\log N}{\log\log N}$. (Recall that, in contrast, a random $3$--regular graph on $N$ vertices
does not embed geometrically into a Euclidean space of dimension $o(\log N)$.) On the other hand, a slight adaptation of a volumetric argument from the same paper~\cite{FM86} (which we reproduce in Section \ref{sec4} for the reader's convenience) yields that $T$~is~not geometrically embeddable into {\it any} normed space of dimension less than $c\frac{\log N}{\log\log N}$, for a universal constant~$c>0$.

It is therefore natural to ask if the ``transition window'' $\Theta\big(\frac{\log N}{\log\log N}\big)$ for complete binary $N\text{--vertex}$ trees is universal across all norms. We completely resolve this problem and, moreover, provide embedding results for arbitrary bounded degree trees.

\begin{theorem}[Universal threshold for embedding of trees] \label{thm:complete-tree-embedding}
There exist absolute positive constants $\alpha_0<\alpha_1$, and for every integer $\Delta\geqslant 3$ a positive integer $N_0=N_0(\Delta)$
depending only on $\Delta$ with the following property. Let $X$ be a finite-dimensional normed space, and let $N\geqslant N_0$ be an integer.
\begin{enumerate}
\item[$(A)$] If\, $\dim (X)\leqslant \alpha_0\, \frac{\log N}{\log\log N}$, then there exists a tree $T$ on $N$ vertices
with maximum degree at most $\Delta$ that is not geometrically embeddable into $X$.
\item [$(B)$] On the other hand, if\, $\dim (X)\geqslant \alpha_1\, \frac{\log N}{\log\log N}$, then every tree $T$ on $N$ vertices
with maximum degree at most $\Delta$ is geometrically embeddable into $X$.
\end{enumerate}
\end{theorem}

\begin{remark} \label{rem1.1}
The proof of Theorem \ref{thm:complete-tree-embedding} yields the estimate $\frac12 \leqslant \alpha_0 <\alpha_1\leqslant 64$.
\end{remark}

Let us discuss the result in the context of intersection graphs. Given a collection of parallel translates $\{x_i+K\colon i\in I\}$ of a convex body $K$ in $\R^m$, the {\it intersection graph} is formed by placing an edge between $i$ and $j$ for all pairs of translates $x_i+K$ and $x_j+K$ that have a nonempty intersection. Our result implies that for any centrally-symmetric convex body $K$, every bounded degree $N$--vertex tree can be realized as an intersection graph in an $\big\lceil\alpha_1\, \frac{\log N}{\log\log N}\big\rceil$--dimensional vector space, using translates of $K$.

\begin{corollary}
Let $\Delta\geqslant 3$ be a positive integer, and let $\alpha_1$ and $N_0$ be as in Theorem~\ref{thm:complete-tree-embedding}. Also let $N\geqslant N_0$ be an integer, let $T$ be a tree on $N$ vertices with maximum degree at most $\Delta$, and let $K$ be any centrally-symmetric convex body in a Euclidean space of dimension $\big\lceil\alpha_1\, \frac{\log N}{\log\log N}\big\rceil$. Then there exists an intersection graph for translates of $K$ which is isometric to $T$.
\end{corollary}

\begin{remark} \label{rem1.3}
The geometric embedding $\zeta\colon V(T)\to X$ obtained in the proof of part $(B)$ of Theorem~\ref{thm:complete-tree-embedding} is $2$-Lipschitz when the tree $T$ is equipped with the shortest-part distance $\dist_T(\cdot,\cdot)$, and it satisfies
$\|\zeta(u)-\zeta(v)\|_X \gtrsim \sqrt[10]{\dist_T(u,v)}$ for all $u,v\in V(T)$. That is, the inverse of $\zeta$ is $\frac{1}{10}\text{-H\"{o}lder}$. This should be compared with a well-known result of Bourgain \cite{Bou86} (see, also, \cite{Ma99}) asserting that any embedding of the complete rooted binary tree on $N$ vertices into a Hilbert space\footnote{In fact, binary trees do not embed with bounded bi-Lipschitz distortion into any super-reflexive spaces (e.g., all $\ell_p$ spaces for $1<p<\infty$.)}
necessarily has bi-Lipschitz distortion at least $\Omega\big(\sqrt{\log\log N}\big)$.\!
\end{remark}

We conclude with some open research questions, followed by a proof overview.

\begin{problem}[Optimal constants in Theorem~\ref{thm:complete-tree-embedding}]
Determine optimal values of the constants $\alpha_0$ and $\alpha_1$ in Theorem~\ref{thm:complete-tree-embedding}. Are they equal up to $o(1)$ additive term?
\end{problem}

\begin{problem}[Trees of growing maximal degree]
Determine an explicit function $m := m(N,\Delta)$ and some universal constants $\alpha_0$ and $\alpha_1$ so that the following holds. Any tree on $N$ vertices with maximum degree $\Delta$ embeds into \textit{any} normed space $X$ of dimension at least $\alpha_1\, m$. And, no complete tree of degree $\Delta$ embeds into any space of dimension less than $\alpha_0\, m$.
\end{problem}

\subsection*{Proof overview}

Part (A) of Theorem~\ref{thm:complete-tree-embedding} is essentially due to Frankl and Maehara. The argument is volumetric: in low dimension there is insufficient volume to pack the required number of disjoint translates of a fixed fraction of the unit ball. An explicit witness is given by the complete rooted $(\Delta-1)$--regular tree of appropriate height.

We now outline the proof of part (B), which is our main contribution. Let $X=(\mathbb{R}^m,\|\cdot\|_X)$ be a normed space of dimension $m$, and let $T$ be a tree on $N$ vertices with maximum degree at most~$\Delta$. We construct a randomized embedding $\boldsymbol{\zeta}\colon V(T)\to X$ in Section~\ref{subsection_Constr_Rand_Emb} as follows. To each edge $e$ of $T$ we assign an independent random vector $\mathbf{Y}_e$, drawn uniformly from the unit ball $B_X$. After fixing a root $\mathrm{o}$, we map each vertex $v$ to the sum of the vectors along the unique path from $\mathrm{o}$ to $v$. In addition, we add a small independent Gaussian perturbation along each edge. Thus,
\[ \boldsymbol{\zeta}(v) :=
\sum_{e\in P(\mathrm{o},v)} \mathbf{Y}_e \;+\; \text{(Gaussian regularization)}. \]
Our goal is to show that with positive probability this embedding is \emph{geometric}:  two vertices are adjacent in $T$ if and only if their images are at spatial distance at most one.  The upper bound for adjacent pairs is straightforward. The main difficulty is proving that for every non-adjacent $\{u,v\}$,
\[ \| \boldsymbol{\zeta}(u)-\boldsymbol{\zeta(v)}\|_X > 1. \]
We refer to the complementary event—namely that a non-adjacent pair is mapped to spatial distance at most one—as the \emph{bad event} associated with $\{u,v\}$.

If $u$ and $v$ have graph distance $\ell$, then their spatial difference is a sum of  $\ell$ independent edge vectors, together with a Gaussian contribution. Thus, the problem reduces to controlling lower deviations of sums of independent log-concave vectors. The available anti-concentration bounds depend naturally on $\ell$, which leads to a division into three regimes, defined in \eqref{eq:009}.

For \textbf{small} $\ell$, we use a thin-shell estimate for isotropic log-concave distributions that follows from \cite{Kl07}; see Lemma \ref{lem:prob_bound_2}. In this regime the dependency structure is sparse: the bad event for a pair $\{u,v\}$ depends only on the random edge vectors along the path $P(u,v)$. Two such events are independent unless the corresponding paths intersect. The resulting dependency graph is sufficiently sparse that the Lovász local lemma applies, allowing us to rule out all short-distance bad events simultaneously. Note that avoidance of a union bound is crucial: there are potentially of order $N$ pairs of vertices at distance two, which is super-exponential in the dimension $m$, yet the probability of the event $\{\|\boldsymbol{\zeta}(u) -\boldsymbol{\zeta}(v)\|_X < 1\}$ decays only inverse-polynomially in the dimension.

For \textbf{intermediate} $\ell$, we invoke the stronger small-ball bounds  for sums of independent log-concave vectors provided by Lemma \ref{lem:prob_bound_1}, which follow from the recent resolution of the hyperplane slicing conjecture \cite{KL24b}. Compared to the regime of small $\ell$, this improved estimate compensates for the larger dependency neighborhoods that arise in the local lemma argument.

Finally, for \textbf{large} $\ell$, lower-deviation estimates on sums of independent uniform vectors suitable for this scaling are not readily available, so we instead rely solely on anti-concentration estimates for the Gaussian regularization.  Existing small-ball estimates conveniently provide the required super-exponential lower-tail decay needed to directly apply a naive union bound over all large-distance bad events.

Crucially, stochastic independence of the Gaussian regularization from the uniform edge vectors allows the different scales to be decoupled: the local lemma argument depends only on the uniform edge variables, while the union bound for large graph distances depends only on the Gaussian part. Combining the local lemma for short and intermediate scales with the Gaussian-based union bound at large scales allows us to conclude that, with positive probability, the embedding is geometric.

\section{Background material} \label{sec2}

\subsection{Notation}

All graphs in this paper have no multiple edges and no self-loops. Given a graph~$G$, we denote its vertex set by $V(G)$ and its edge set by $E(G)$. For every $v\in V(G)$, the \emph{neighbourhood of $v$ in $G$} is denoted by $N_G(v):= \big\{u\in V(G)\colon \{u,v\}\in E(G)\big\}$, and $\mathrm{dist}_G(\cdot,\cdot)$ denotes the shortest path graph distance on $V(G)$. We recall that a \emph{tree} is a connected graph with no cycles. We also recall that a \emph{normed space} $(X,\|\cdot\|)$ is a vector space $X$ equipped with a norm; for a normed space $(X,\|\cdot\|)$, by $B_X:=\big\{x\in X\colon \|x\|\leqslant 1\big\}$ we denote its closed unit ball.

\begin{definition} \label{def2.1}
Let $T$ be a tree, and let $u,v\in V(T)$. We say that a set $P\subseteq E(T)$ \emph{connects $u$ and $v$} if there exist $u_0,\ldots,u_p\in V$, where $p=|P|$, such that $u_0=u$, $u_p=v$ and $P=\big\{\{u_{i-1},u_i\}\colon i\in[p]\big\}$.

If $u\neq v$, then $P(u,v)$ denotes the smallest, under inclusion, set of edges that connects $u$ and $v$; by convention, we set $P(u,v)=\emptyset$ if $u=v$. We shall refer to sets of the form $P(u,v)$ as \emph{paths}.
\end{definition}

\subsection{Isotropic random vectors, and normed spaces in isotropic position}

Recall that a random vector $\mathbf{U}$ in $\mathbb{R}^m$ is \textit{isotropic} if its components are mean-zero random variables, and its covariance matrix is the identity.

Now, let $X=(\mathbb{R}^m,\|\cdot\|_X)$ be an $m$-dimensional normed space, and let $\mathbf{X}$ be a random vector in~$\mathbb{R}^m$ uniformly distributed in $B_X$. We say that $X$ is in \textit{isotropic position} if the random vector $\mathbf{X}$ is isotropic;
we emphasize that in asymptotic geometric analysis literature, an alternative definition involving normalization of the Lebesgue volume of $B_X$ is used occasionally. It is a standard fact that any finite-dimensional normed space can be placed in isotropic position by applying an invertible linear transformation
(see, e.g., \cite{KL24a}).

\subsection{Asymptotic geometric analysis}

We will need two deep results from asymptotic geometric analysis. The first one is the following thin-shell estimate due to Klartag \cite{Kl07}; we note that optimal results in this direction were very recently obtained by Klartag--Lehec \cite{KL25}.

\begin{theorem}[Thin-shell concentration for log-concave measures] \label{thm:thinshell}
There exist absolute constants $0<\kappa <1$ and $C>0$ with the following property. If\, $\mathbf{U}$ is an isotropic random vector in $\mathbb{R}^m$ with a log-concave\footnote{A function $f\colon \mathbb{R}^m\to [0,\infty)$ is \emph{log-concave} if, setting $K:=\{x\in\mathbb{R}^m\colon f(x) > 0\}$, we have that the function $\log f\colon K\to\mathbb{R}$ is concave. (Note that the set $K$ is convex.)} distribution, then
\begin{equation} \label{eq2.2}
\mathbb{P}\Big(\big|\|\mathbf{U}\|_2 - \sqrt{m}\big|  \geqslant m^{\frac12 - \kappa} \Big) \leqslant C \exp\big(-m^{\kappa}\big).
\end{equation}
\end{theorem}

We will also need the following very recent result of Klartag and Lehec \cite{KL24b} that yields an affirmative solution to the celebrated Bourgain's slicing conjecture.

\begin{theorem}[Anticoncentration for log-concave measures via slicing] \label{thm:slicing}
There exists an absolute constant $L>0$ with the following property. If\, $\mathbf{U}$ is an isotropic random vector in $\mathbb{R}^m$ with a log-concave density $f$, then for any $x \in \mathbb{R}^m$,
\begin{equation} \label{eq2.3}
f(x) \leqslant L^{m}.
\end{equation}
\end{theorem}

\begin{remark}
Thin-shell estimates for log-concave distributions have been actively studied in the asymptotic geometric analysis literature, leading to strengthening of the result in \cite{Kl07}, as a consequence of the tremendous recent progress on the {\it KLS conjecture} and the aforementioned resolution of the slicing conjecture. We do not attempt to give here a thorough overview of the history of these problems, and refer instead to recent works \cite{Bi25, Ch21, KL24a, Gu24, KL24b} for further references. Our proof of the main result is not sensitive to the value of the constant $\kappa$ in Theorem~\ref{thm:thinshell}.
\end{remark}

\subsection{Gaussian concentration}

We will need the following simple consequence of gaussian concentration for Lipschitz functions.
\begin{fact}\label{cor:gaussian_tails}
If\, $\mathbf{G}$ is a random vector in $\mathbb{R}^m$ with i.i.d. standard normal entries, then for~any~$t>\sqrt{m}$,
\begin{equation} \label{eq2.4}
\mathbb{P}\big( \|\mathbf{G}\|_2>t \big) \leqslant 2\exp \Big( -\frac{(t-\sqrt{m})^2}{2} \Big).
\end{equation}
\end{fact}
\begin{proof}
Notice that the map $f\colon \mathbb{R}^m\to \mathbb{R}$ defined by $f(x) = \|x\|_2$, is $1$-Lipschitz, and it satisfies $\mathbb{E}\big[f(\mathbf{G})\big] \leqslant \sqrt{m}$. Using these observations, the result follows from the standard gaussian concentration for Lipschitz functions (see, e.g., \cite[Sections 1.1--1.2]{Le01}).
\end{proof}

\subsection{John's theorem}

Recall that if $X$ and $Y$ are two normed spaces of the same dimension, then their \emph{Banach--Mazur distance} is defined by
\[ d_{\mathrm{BM}}(X,Y):= \inf\big\{ \|T\|\cdot\|T^{-1}\|\colon T\colon X\to Y \text{ is an isomorphism}\big\}. \]
Also recall that, by a classical result due to John (see, e.g., \cite{JL01}), we have
$d_{\mathrm{BM}}(X,\ell_2^m)\leqslant \sqrt{m}$ for any $m$-dimensional normed space $X$.
We will use this estimate in the following form.
\begin{fact} \label{lem:John_position}
Let $X=(\mathbb{R}^m,\|\cdot\|_X)$ be an $m$-dimensional normed space. Then there exist vectors
$x_1,\dots,x_m\in \mathbb{R}^m$ such that for any $a_1,\dots,a_m\in \mathbb{R}$,
\begin{equation} \label{eq2.5}
\Big(\sum_{i=1}^{m}a_i^2\Big)^\frac{1}{2} \leqslant
\Big\| \sum_{i=1}^{m}a_i x_i\Big\|_X \leqslant \sqrt{m} \Big(\sum_{i=1}^{m}a_i^2\Big)^\frac{1}{2}.
\end{equation}
\end{fact}

\subsection{Lov\'{a}sz local lemma}

The last result we will need is the so-called asymmetric version of Lov\'{a}sz local lemma (see, e.g., \cite[Lemma 5.1.1]{AS16}).

\begin{lemma}[Asymmetric Lov\'{a}sz local lemma] \label{lem:LLL}
Let $G$ be a graph, and let $(A_v)_{v\in V(G)}$ be a collection of measurable events in some probability space such that for every $v\in V(G)$, the event $A_v$ is independent of the events $\big\{A_u\colon u\not\in N_G(v)\cup\{v\}\big\}$. Assume that there exists $h\colon V\to [0,1)$ such that for any $v\in V(G)$,
\begin{equation} \label{eq2.6}
\mathbb{P}(A_v) \leqslant h(v) \prod_{u\in N_G(v)}\big(1-h(u)\big).
\end{equation}
Then,
\begin{equation} \label{eq2.7}
\mathbb{P}\Bigg( \bigcap_{v\in V(G)} A_v^{\complement} \Bigg) \geqslant \prod_{v\in V(G)} \big(1-h(v)\big).
\end{equation}
\end{lemma}

\section{Preparatory lemmas} \label{sec3}

In this section we shall present three preparatory lemmas that are used in the proof of Theorem~\ref{thm:complete-tree-embedding}. We start with the following volumetric estimate.

\begin{lemma} \label{lem:ball_volume}
There exists an absolute constant $C_B>0$ such that for any $m$-dimensional normed space $X=(\mathbb{R}^m,\|\cdot\|_X)$ in isotropic position, we have $\lambda(B_X)\leqslant C_B^m$, where $\lambda$ denotes the Lebesgue measure in $\mathbb{R}^m$.
\end{lemma}
\begin{proof}
Let $\mathbf{U}$ be a random vector uniformly distributed in $B_X$. Then $\mathbf{U}$ is an isotropic, log-concave, random vector in $\mathbb{R}^m$ with density function $f := \lambda(B_X)^{-1}\mathbbm{1}_{B_X}$. By Theorem \ref{thm:thinshell}, there exists an absolute constant $c>0$ such that
\[ \frac{1}{2}\leqslant \mathbb{P}\big( \|\mathbf{U}\|_2\leqslant c\sqrt{m}\big) =
\frac{\lambda\big(B_X\cap c\sqrt{m}\,B_{\ell^m_2}\big)}{\lambda(B_X)}. \]
Therefore,
\[ \lambda(B_X) \leqslant 2\lambda\big(B_X\cap c\sqrt{m}\, B_{\ell^m_2}\big) \leqslant
2\, c^m\, \lambda\big(\sqrt{m}\, B_{\ell^m_2}\big) \leqslant 2 \big(\sqrt{2c^2\pi e}\big)^m. \qedhere \]
\end{proof}

The next lemma is, in effect, a small-ball probability estimate with the ball of a given finite-dimensional normed space $X$ in place of the usual Euclidean ball.

\begin{lemma} \label{lem:prob_bound_1}
Let $X=(\mathbb{R}^m,\|\cdot\|_X)$ be a normed space in isotropic position, let $k$ be a positive integer, and let $0<c_1<\frac{1}{2}$ such that
\begin{equation} \label{eq:004}
k\geqslant \big(C_B\,L\big)^{\frac{4}{1-2c_1}},
\end{equation}
where $C_B$ is as in Lemma \ref{lem:ball_volume}, and $L$ is as in Theorem \ref{thm:slicing}. Finally, let $(\mathbf{Y}_1,\dots,\mathbf{Y}_k)$ be i.i.d. random vectors uniformly distributed in $B_X$. Then,
\begin{equation} \label{eq3.1}
\mathbb{P} \Bigg( \Big\|  \sum_{i=1}^{k} \mathbf{Y}_i \Big\|_X \leqslant k^{c_1} \Bigg)
\leqslant \exp\Big(-\frac{1}{2} \Big(\frac{1}{2}-c_1\Big) m \log k \Bigg).
\end{equation}
\end{lemma}

\begin{proof}
Let $g$ denote the density function of the random vector $\mathbf{Y}=\frac{1}{\sqrt{k}}\sum_{i=1}^{k}\mathbf{Y}_i$. Since the convolution of two log-concave distributions is log-concave---see, e.g., \cite[Section 7.2]{Ga02}---we see that the random vector $\mathbf{Y}$ is log-concave; it is also clear that $\mathbf{Y}$ is isotropic. Consequently, by Theorem~\ref{thm:slicing}, the function $g$ is pointwise bounded by~$L^m$. Hence, by Lemma \ref{lem:ball_volume}, we have
\begin{align} \label{eq3.2}
\mathbb{P} \Bigg( \Big\|  \sum_{i=1}^{k} \mathbf{Y}_i \Big\|_X \leqslant k^{c_1} \Bigg) & =
\mathbb{P} \Bigg( \Big\| \frac{1}{\sqrt{k}} \sum_{i=1}^{k} \mathbf{Y}_i \Big\|_X \leqslant k^{c_1-\frac{1}{2}} \Bigg) \\
& = \int g(x) \mathbbm{1}_{k^{c_1-\frac{1}{2}}B_X}(x)\, dx
\leqslant L^m\, \lambda\big(k^{c_1-\frac{1}{2}} B_X\big) \nonumber \\
& \leqslant \Big(\frac{C_B L}{k^{\frac{1}{2}-c_1}}\Big)^m \leqslant
\exp\Big(-\frac{1}{2} \Big(\frac{1}{2}-c_1\Big)m \log k \Big). \nonumber \qedhere
\end{align}
\end{proof}

The next lemma complements Lemma \ref{lem:prob_bound_1}, and it will used when $k$ is smaller than the threshold appearing in the right-hand-side of \eqref{eq:004}.

\begin{lemma} \label{lem:prob_bound_2}
There exists a positive integer $m_1$ with the following property. Let $m\geqslant m_1$ be an integer, let $X=(\mathbb{R}^m,\|\cdot\|_X)$ be a normed space in isotropic position, let $k\geqslant 2$ be an integer, and let $(\mathbf{Y}_1,\dots,\mathbf{Y}_k)$ be i.i.d. random vectors uniformly distributed in $B_X$. Then,
\begin{equation} \label{eq3.4}
\mathbb{P} \Bigg( \Big\|  \sum_{i=1}^{k} \mathbf{Y}_i \Big\|_X \leqslant 1+\frac{1}{2\,m^{1-\kappa}} \Bigg)
\leqslant \exp\big(-m^{\frac{\kappa}{2}}\big),
\end{equation}
where $\kappa$ is as in Theorem \ref{thm:thinshell}.
\end{lemma}

\begin{proof}
Set $\gamma := \frac{2-\sqrt2}{4}$, $\varepsilon:=\frac{1}{2\,m^{1-\kappa}}$, and
$\Gamma:= B_X\setminus \Big((1+\varepsilon)^{-1}(1-\gamma)\sqrt{k}\sqrt{m}\, B_{\ell_2^m}\Big)$. Notice that
\begin{equation}\label{eq:005}
\mathbb{P}\Bigg( \Big\| \sum_{i=1}^{k} \mathbf{Y}_i \Big\|_X \leqslant 1+\varepsilon \Bigg) \leqslant
\mathbb{P}\Big( \sum_{i=1}^{k} \mathbf{Y}_i \in (1-\gamma)\sqrt{k}\sqrt{m}\,B_{\ell^m_2} \Big) +
\mathbb{P}\Big(\sum_{i=1}^{k} \mathbf{Y}_i \in (1+\varepsilon)\Gamma\Big).
\end{equation}
As in the proof of Lemma \ref{lem:prob_bound_1}, we observe that the random vector $\frac{1}{\sqrt{k}} \sum_{i=1}^{k} \mathbf{Y}_i $ is isotropic and log-concave. Therefore, by Theorem \ref{thm:thinshell}, if $m_1$ is sufficiently large, we have
\begin{equation} \label{eq:006}
\mathbb{P}\Bigg(\sum_{i=1}^{k} \mathbf{Y}_i \in (1-\gamma)\sqrt{k}\sqrt{m}\,B_{\ell_2^m} \Bigg) =
\mathbb{P}\Bigg(\Big\|\frac{1}{\sqrt{k}}\sum_{i=1}^{k} \mathbf{Y}_i \Big\|_2\leqslant (1-\gamma)\sqrt{m} \Bigg)
\leqslant \frac{1}{2}\exp\big(-m^{\frac{\kappa}{2}}\big).
\end{equation}
Notice that
$(1-\gamma)\sqrt{2}=\frac{\sqrt{2}+1}{2}>1$.
Hence, by Theorem \ref{thm:thinshell} and assuming that $m_1$ is sufficiently large, we get that
$$ \lambda(\Gamma) \leqslant C\exp\big(-m^{\kappa}\big)\,\lambda(B_X). $$
It follows immediately that, conditioned on any realization of $\mathbf{Y}_1,\dots,\mathbf{Y}_{k-1}$, the probability that $\sum_{i=1}^{k} \mathbf{Y}_i$ belongs to $(1+\varepsilon)\Gamma$, is bounded above by
$$ \frac{\lambda\big((1+\varepsilon)\Gamma\big)}{\lambda(B_X)} \leqslant
C(1+\varepsilon)^m\exp\big(-m^{\kappa}\big) \leqslant \frac{1}{2}\exp\big(-m^{\frac{\kappa}{2}}\big), $$
where we have used again the fact that $m_1$ can be taken sufficiently large. Combining the last estimate with \eqref{eq:005}--\eqref{eq:006}, we get the result.
\end{proof}

\section{Proof of Theorem \ref{thm:complete-tree-embedding}} \label{sec4}

As mentioned, part $(A)$ of Theorem \ref{thm:complete-tree-embedding} is due to Frankl and Maehara \cite{FM86}, and it is essentially a volumetric argument: small dimension does not have enough volume to fit in all the norm balls needed. For the convenience of the reader, we recall the proof. We set
\begin{equation} \label{pa.eq.1}
\alpha_0:= \frac12.
\end{equation}
Fix an integer $\Delta\geqslant 3$, and let $N$ be an arbitrary positive integer that is sufficiently large in terms of $\Delta$. We define a tree $T$ on $N$ vertices with maximum degree $\Delta$ as follows. There exists a root $r_T\in V(T)$ such that, setting $h_0:=\big\lceil \log\big((N-1)\frac{\Delta-2}{\Delta}+1\big)/\log(\Delta-1)\big\rceil$, the following hold.
\begin{enumerate}
\item[$\bullet$] For every $v\in V(T)$, we have $\mathrm{dist}_T(r_T,v)\leqslant h_0$.
\item[$\bullet$] For every $v\in V(T)$ with $\mathrm{dist}_T(r_T,v)\leqslant h_0-2$, we have $|N_T(v)|=\Delta$.
\end{enumerate}
Notice, in particular, that for $n \geqslant 1$ and $N = 1 + \Delta \frac{(\Delta-1)^{n}-1}{\Delta-2}$, then $T$ is just the $(\Delta-1)$-regular tree of height $n$. Let $X=(\mathbb{R}^m,\|\cdot\|_X)$ be a normed space for which there exists a geometric embedding $\zeta\colon V(T)\to X$. Since $T$ is a bipartite graph (being a tree), there exists an independent set $V_0\subseteq V(T)$ with $|V_0|\geqslant\frac{N}{2}$. Using the fact that every node of $T$ has distance from $r_T$ at most $h_0$, we see that
\begin{equation}\label{eq:038}
\bigcup_{v\in V_0}B_X\big(\zeta(v),1/2\big) \subseteq B_X\big(\zeta(r_T), h_0+1/2\big).
\end{equation}
On the other hand, since $\zeta$ is a geometric embedding, for every $u,v\in V_0$ with $u\neq v$ we have that $B_X\big(\zeta(u),1/2\big)\cap B_X\big(\zeta(v),1/2\big)=\emptyset$. Thus, denoting by $\lambda$ the Lebesgue measure in $\mathbb{R}^m$, by \eqref{eq:038},
\[ \frac{N}{2}\lambda\big(B_X(0,1/2)\big) \leqslant \lambda\big(B_X(0,h_0+1/2)\big), \]
which is easily seen to imply, if $N$ is sufficiently large, that $\dim(X)=m > \alpha_0 \frac{\log N}{\log\log N}$, as desired.

We proceed to the proof of part $(B)$. We shall construct a random embedding of a given tree $T$ on $N$ vertices
with maximum degree $\Delta$ in the normed space $X=(\mathbb{R}^m,\|\cdot\|_X)$, and we shall show that with positive probability
this is a geometric embedding. To that end we start by setting
\begin{equation} \label{eq:001}
\alpha_1:=64, \ \ \ \ c_1:=\frac{1}{4}, \ \ \ \ c_2:=\min\Big\{\frac{1}{4},\frac{\kappa}{4}\Big\},
\end{equation}
where $\kappa$ is as in Theorem \ref{thm:thinshell}. The parameter $N_0$ will be determined in the course
of the proof. For notational convenience, we shall denote by $V$ the vertex set of $T$, and by $E$ its edge set.
Clearly, we may assume that $\dim(X)=m= \big\lceil \alpha_1 \log N/\log\log N \big\rceil$ and that $X$ is in isotropic position.

\subsection{Construction of the random embedding} \label{subsection_Constr_Rand_Emb}

Let $(\mathbf{Y}_e)_{e\in E}$ be a collection of i.i.d. random vectors uniformly distributed in $B_X$.
Also let $(\mathbf{g}_i^e)_{i\in [m], e\in E}$ be a collection of i.i.d. scalar standard normal random variables that are
independent of $(\mathbf{Y}_e)_{e\in E}$. (As usual, $[m]:=\{1,\dots,m\}$ denotes the discrete interval of length $m$.)
Let $x_1,\ldots,x_m\in \mathbb{R}^m$ be as in Fact \ref{lem:John_position}, and define for every~$e\in E$,
\begin{equation} \label{eq:G-def}
    \mathbf{G}_e := \sum_{i=1}^{m}\mathbf{g}_i^e\, x_i.
\end{equation}
Fix a distinguished vertex $r_T\in V$ that we view as the root of $T$, and set
\begin{equation} \label{eq:003}
\delta := \exp\Big( -\Big(\frac{1}{2}-\frac{c_1}{2}\Big) \log^{c_2}N \Big).
\end{equation}
We define the random embedding $\boldsymbol{\zeta}\colon V\to X$ by
\[ \boldsymbol{\zeta}(v) :=
\Big( 1-\frac{1}{\log N} \Big) \sum_{e\in P(r_T, v)} \mathbf{Y}_e +\delta \sum_{e\in P(r_T, v)} \mathbf{G}_e  \]
with the convention that the sum over the empty set is equal to the zero vector. (Here, $P(r_T,v)$ is as in Definition \ref{def2.1}.) As we shall see the parameter $\delta$---and the constant $c_2$---are chosen to moderate the contribution of the Gaussian part of the random embedding $\boldsymbol{\zeta}$ in a certain regime.

Next, for every $u,v\in V$ with $u\neq v$, define the ``good'' event
\[ \mathcal{A}_{u,v}:=
\begin{cases}
\big[ \|\boldsymbol{\zeta}(u)-\boldsymbol{\zeta}(v) \|_X \leqslant 1 \big] & \text{if } \, \{u,v\}\in E, \\
\big[ \| \boldsymbol{\zeta}(u)-\boldsymbol{\zeta}(v) \|_X > 1 \big] & \text{if } \, \{u,v\}\not\in E.
\end{cases} \]
Part $(B)$ of Theorem \ref{thm:complete-tree-embedding} will follow once we show that
\begin{equation} \label{eq:008}
\mathbb{P}\Bigg( \bigcap_{\substack{u,v\in V \\ u\neq v}} \mathcal{A}_{u,v} \Bigg)>0.
\end{equation}

\subsection{Auxiliary events}

Set
\begin{equation} \label{eq:009}
\ell_0 := \exp\big( \log^{c_2}N\big) \ \ \ \text{ and } \ \ \
k_0:= \max\Big\{ 3,\big\lceil (C_BL)^{4/(1-2c_1)}\big\rceil, \big\lceil 4^{1/c_1}\big\rceil \Big\},
\end{equation}
where $C_B$ is as in Lemma \ref{lem:ball_volume}, and $L$ is as in Theorem \ref{thm:slicing}.
Taking $N_0$ sufficiently large, we may assume that $k_0\leqslant \ell_0$.

Next, for every $u,v\in V$ with $2\leqslant \mathrm{dist}_T(u,v) \leqslant \ell_0$, define the event
\[ \mathcal{L}_{u,v}:=
\begin{cases}
\Big[ \big\| \underset{e\in P(r_T,u)}{\sum} \mathbf{Y}_e - \underset{e\in P(r_T,v)}{\sum} \mathbf{Y}_e \big\|_X
\leqslant 1+\frac{1}{2\, m^{1-\kappa}} \Big] & \text{if } \, 2\leqslant \mathrm{dist}_T(u,v) \leqslant k_0, \\
\Big[ \big\| \underset{e\in P(r_T,u)}{\sum} \mathbf{Y}_e - \underset{e\in P(r_T,v)}{\sum} \mathbf{Y}_e\big\|_X
\leqslant \mathrm{dist}_T(u,v)^{c_1} \Big] & \text{if } \, k_0<\mathrm{dist}_T(u,v) \leqslant \ell_0,
\end{cases} \]
where $\kappa$ is as in Theorem \ref{thm:thinshell}. Our goal in this subsection is to show that
\begin{equation} \label{eq:010}
\mathbb{P}\Bigg( \bigcap_{\substack{ u,v\in V \\ 2\leqslant\mathrm{dist}_T(u,v)\leqslant\ell_0 }}
\mathcal{L}_{u,v}^{\complement}\Bigg)>0.
\end{equation}

\subsubsection{The ``path'' graph}

Define the ``path" graph $\mathcal{G} = (\mathcal{V},\mathcal{E})$, which will act as the dependency graph in our application of the Lov\'{a}sz local lemma, by setting
\[ \mathcal{V}:=\big\{\{u,v\}\colon u,v\in V \text{ and } 2\leqslant\mathrm{dist}_T(u,v)\leqslant\ell_0\big\}, \]
\[ \mathcal{E}:=\Big\{ \big\{\{u,v\},\{u',v'\}\big\}\colon \{u,v\},\{u',v'\}\in \mathcal{V} \text{ and }
P(u,v)\cap P(u',v') \neq \emptyset\Big\}. \]
Moreover, for every $k\in \{2,\dots, \ell_0\}$ set
$\mathcal{V}_k:= \big\{ \{u,v\}\in\mathcal{V}\colon \mathrm{dist}_T(u,v)=k\big\}$.
\begin{claim} \label{cl:control_degree}
Let $k,\ell\in \{2,\dots,\ell_0\}$, and let $\{u,v\}\in\mathcal{V_\ell}$. Then,
\[ |N_\mathcal{G}(\{u,v\})\cap  \mathcal{V}_k|\leqslant \min\big\{N^2, 5\ell k (\Delta-1)^{k-1}\big\}. \]
\end{claim}
\begin{proof}
Clearly, $|N_\mathcal{G}(\{u,v\})\cap \mathcal{V}_k|$ is upper bounded by $N^2$. Thus, it suffices to show that
\begin{equation} \label{eq:011}
|N_\mathcal{G}(\{u,v\})\cap  \mathcal{V}_k|\leqslant 5\ell k (\Delta-1)^{k-1}.
\end{equation}
Let $s\in [k]$ be arbitrary, and set $a_s := \big|\big\{\{u',v'\}\in N_\mathcal{G}(\{u,v\})\cap  \mathcal{V}_k\colon |P(u,v)\cap P(u',v')|=s\big\}\big|$.

We shall consider various cases. First, assume that $s>\min\{\ell,k\}$; then, clearly,
\begin{equation}\label{eq:012}
a_s=0.
\end{equation}
Next, assume that $s=\ell\leqslant k$; then, since $\ell,k\geqslant 2$, we have
\begin{equation}\label{eq:013}
a_\ell \leqslant \sum_{x=0}^{k-\ell}(\Delta-1)^x(\Delta-1)^{k-\ell-x} = (\Delta-1)^{k-\ell}(k-\ell+1)
\leqslant \frac{\ell k}{4}(\Delta-1)^{k-1}.
\end{equation}
If $s = \ell-1 \leqslant k$, then
\begin{equation} \label{eq:014}
a_s \leqslant
2\big( (\Delta-1)^{k-\ell+1} + (\Delta-2)(\Delta-1)^{k-\ell}(k-\ell+1) \big)
\leqslant \ell k (\Delta-1)^{k-1}.
\end{equation}
Finally, if $s \leqslant \min\{\ell-2 , k\}$, then
\begin{align} \label{eq:015}
a_s & \leqslant 2 \big( (\Delta-1)^{k-s} +(\Delta -2)(\Delta-1)^{k-s-1}(k-s)\big) + \\
& \ \ \ \ \ \ \ \ \ \ \
+ (\ell-s-1) \big( 2(\Delta-2)(\Delta-1)^{k-s-1} +(\Delta-2)^2(\Delta-1)^{k-s-2}(k-s-1)  \big) \nonumber \\
& \leqslant 3\ell k(\Delta-2)(\Delta-1)^{k-s-1}. \nonumber
\end{align}
By \eqref{eq:012} and \eqref{eq:015}, we have
\begin{equation} \label{eq:016}
\sum_{s=1}^{\ell-2}a_s\leqslant 3\ell k(\Delta-1)^{k-1}.
\end{equation}
The desired estimate \eqref{eq:011} follows from \eqref{eq:012}--\eqref{eq:014} and \eqref{eq:016}.
\end{proof}

\subsubsection{Application of asymmetric Lovász local lemma}

Set
\[ k_1 := \Big\lfloor \frac{\log N }{\log(\Delta-1)} \Big\rfloor; \]
since $N_0$ can be taken sufficiently large, we may assume that for every $k\in\{2,\dots, k_1\}$ we have
\begin{equation}\label{eq:017}
6k\ell_0(\Delta-1)^k\leqslant N^2.
\end{equation}
For every $k\in \{2,\dots,\ell_0\}$ set
\[ b_k:=
\begin{cases}
\frac{1}{2^{k+1}6\ell_0 k (\Delta-1)^k}
 & \text{if } \, 2\leqslant k\leqslant k_1, \\
\frac{3}{2\pi^2 k^2 N^2}
& \text{if } \, k_1<k \leqslant \ell_0,
\end{cases} \]
and define $h\colon \mathcal{V}\to[0,1)$ by $h(\{u,v\}) = b_{\mathrm{dist}_T(u,v)}$.
By the choice of $(b_k)_{k=2}^{\ell_0}$ and \eqref{eq:017}, we see that
$\sum_{k=2}^{\ell_0} b_k\cdot\min\{N^2, 6\ell_0 k (\Delta-1)^k \}\leqslant \frac{1}{2}$.
Hence, by Claim \ref{cl:control_degree}, for every $\{u,v\}\in \mathcal{V}$,
\begin{equation}\label{eq:018}
\prod_{\{u',v'\}\in N_{\mathcal{G}}(\{u,v\})} \big(1-h(\{u',v'\})\big)\geqslant
\prod_{k=2}^{\ell_0} (1-b_k)^{\min\{N^2, 5\ell k (\Delta-1)^{k-1}\}} \geqslant \frac{1}{2}.
\end{equation}
\begin{claim}\label{claim_Lovatz_condition}
For every $\{u,v\}\in\mathcal{V}$, we have $\mathbb{P}(\mathcal{L}_{u,v})\leqslant\frac{1}{2}h(\{u,v\})$.
\end{claim}
\begin{proof}
Fix $\{u,v\}\in\mathcal{V}$, and set $k:=\mathrm{dist}_T(u,v)$. We distinguish three cases.

First, assume that $k\in \{2,\dots,k_0\}$. Since the random vectors
$\sum_{e\in P(r_T,u)}\mathbf{Y}_e - \sum_{e\in P(r_T,v)}\mathbf{Y}_e$ and
$\sum_{e\in P(u,v)}\mathbf{Y}_e $ have the same distribution, then provided that $N_0$ is sufficiently large,
by Lemma \ref{lem:prob_bound_2},
\begin{equation}\label{eq:019}
\mathbb{P}(\mathcal{L}_{u,v})\leqslant \exp\big(- m ^{\frac{\kappa}{2}}\big) \leqslant
\exp\Big(- \alpha_1^{\frac{\kappa}{2}}\frac{\log^{\frac{\kappa}{2}}N}{(\log\log N)^{\frac{\kappa}{2}}}\Big).
\end{equation}
On the other hand, by the choice of $\ell_0$ in \eqref{eq:009} and the fact that $b_k\geqslant b_{k_0}$ for $k\in \{2,\dots, k_0\}$,
\begin{equation}\label{eq:020}
\frac{h(\{u,v\})}{2} =\frac{b_k}{2} \geqslant \frac{b_{k_0}}{2} \geqslant
\exp\big(-\log 24-k_0\log(2\Delta-2)-\log k_0 - \log^{c_2}N\big).
\end{equation}
Since $c_2<\frac{\kappa}{2}$ and using the fact that $N_0$ can be taken sufficiently large, the claim
follows in this case by~\eqref{eq:019} and \eqref{eq:020}.

Next, assume that $k\in \{k_0+1,\dots,k_1\}$. Using again the fact that
$\sum_{e\in P(r_T,u)}\mathbf{Y}_e - \sum_{e\in P(r_T,v)}\mathbf{Y}_e$ and
$\sum_{e\in P(u,v)}\mathbf{Y}_e $ have the same distribution, by Lemma \ref{lem:prob_bound_1},
the choice of $k_0$ in \eqref{eq:009} and the fact that $c_1\leqslant\frac{1}{4}$, we obtain that
\begin{equation}\label{eq:021}
\mathbb{P}(\mathcal{L}_{u,v}) \leqslant \exp\Big(-\frac{1}{2} \Big(\frac{1}{2}-c_1\Big)m\log k \Big)
\leqslant \exp\Big( -\frac{\alpha_1}{8} \cdot \frac{\log N}{\log\log N} \log k \Big).
\end{equation}
Since the function $f(x) = \frac{x}{\log x}$ is increasing for $x\geqslant3$ and $k>k_0\geqslant3$, we see that
\begin{align*}
\frac{h(\{u,v\})}{2} & =
\exp\big(-\log 24-k\log(2\Delta-2)-\log k - \log^{c_2}N\big) \\
& =\exp\Big(-\Big(\frac{\log 24}{\log k}+\frac{k}{\log k }\log(2\Delta-2)+1+
\frac{\log^{c_2}N}{\log k}\Big)  \log k \Big)\\
& \geqslant \exp\Big(-\Big(\frac{\log 24}{\log 3 }+\frac{k_1}{\log k_1 }\log(2\Delta-2)+1+
\frac{\log^{c_2}N}{\log k }\Big)  \log k \Big)\\
& \geqslant \exp\Big(-\Big(4+\frac{\log N }{\log\log N-\log\log(\Delta-1)}\cdot\frac{\log (2(\Delta-1))}{\log(\Delta-1)}+
\frac{\log^{c_2}N}{\log k}\Big)  \log k \Big).
\end{align*}
Taking $N_0$ sufficiently large and observing that $\frac{\log(2(\Delta-1))}{\log(\Delta-1)}\leqslant2$
and $\frac{\alpha_1}{8}>4$, we conclude that
\begin{equation}\label{eq:022}
\frac{h(\{u,v\})}{2}
\geqslant\exp\Big(-\Big(4+4\frac{\log N}{\log\log N}+\log^{c_2}N\Big)\log k\Big)
\stackrel{\eqref{eq:021}}{\geqslant}\mathbb{P}(\mathcal{L}_{u,v}).
\end{equation}

Finally, assume that $k\in \{k_1+1,\dots,\ell_0\}$. As before, since
$\sum_{e\in P(r_T,u)}\mathbf{Y}_e - \sum_{e\in P(r_T,v)}\mathbf{Y}_e$ and
$\sum_{e\in P(u,v)}\mathbf{Y}_e $ have the same distribution, by Lemma \ref{lem:prob_bound_1},
the choice of $k_0$ in \eqref{eq:009} and the fact that $k>k_0$ and $c_1\leqslant\frac{1}{4}$, we have
\begin{equation}\label{eq:023}
\mathbb{P}(\mathcal{L}_{u,v}) \leqslant \exp\Big(-\frac{1}{2} \Big(\frac{1}{2}-c_1\Big)m\log k \Big)
\leqslant \exp\Big( -\frac{\alpha_1}{8} \cdot \frac{\log N}{\log\log N} \log k \Big).
\end{equation}
Taking $N_0$ sufficiently large, we may assume that
$\log k \geqslant \log \log N - \log\log(\Delta-1) \geqslant \frac{1}{2}\log\log N$.
Thus, by \eqref{eq:023} and the choice of $\alpha_1$ in \eqref{eq:001}, we have that
\begin{equation}\label{eq:024}
\mathbb{P}(\mathcal{L}_{u,v})\leqslant \exp( -4\log N).
\end{equation}
Since $k_1 < k \leqslant \ell_0$, taking $N_0$ sufficiently large, we see that
\begin{equation} \label{eq:025}
\frac{h(\{u,v\})}{2} = \exp\Big(-\log\Big(\frac{4\pi^2}{3}\Big) -2\log k -2\log N\Big)
\geqslant \exp\Big(- 3\log N \Big).
\end{equation}
Thus, this final case the claim follows from  \eqref{eq:024} and \eqref{eq:025}.
\end{proof}
By Claim \ref{claim_Lovatz_condition} and inequality \eqref{eq:018},  for every $\{u,v\}\in \mathcal{V}$,
\[ \mathbb{P}(\mathcal{L}_{u,v})\leqslant h(\{u,v\})
\prod_{\{u',v'\}\in N_{\mathcal{G}}(\{u,v\})} \big(1-h(\{u',v'\})\big). \]
Therefore, inequality \eqref{eq:010} follows from Lemma \ref{lem:LLL}.

\subsection{Completion of the proof of part $(B)$ of Theorem \ref{thm:complete-tree-embedding}}

Let $x_1,\dots,x_m\in\mathbb{R}^m$ be the vectors selected in Subsection \ref{subsection_Constr_Rand_Emb}.
It is convenient to introduce an auxiliary norm $\|\cdot\|_2^*$ in $\mathbb{R}^m$ as follows.
Given $x\in \mathbb{R}^m$ written as $x = \sum_{i=1}^{m} a_i x_i$, we set
$\|x\|_2^* = (a_1^2+\dots+a_m^2)^{1/2}$. By \eqref{eq2.5}, we see that for every $x\in\mathbb{R}^m$,
\begin{equation}\label{eq:026}
  \|x\|_2^* \leqslant \|x\|_X \leqslant \sqrt{m}\|x\|_2^*.
\end{equation}
Moreover, recalling \eqref{eq:G-def} and Fact \ref{cor:gaussian_tails}, for every subset $E'$ of the edge-set $E$ of $T$ and
every $t>\sqrt{m}$,
\begin{equation}\label{eq:027}
    \mathbb{P}\Bigg( \Big\|\frac{1}{\sqrt{|E'|}}\sum_{e\in E'}\mathbf{G}_e\Big\|_2^*>t \Bigg)
    \leqslant 2\exp \Big( -\frac{(t-\sqrt{m})^2}{2} \Big).
\end{equation}

Now, set
\[ \mathcal{L} :=
\bigcap_{\substack{ u,v\in V \\ 2\leqslant\mathrm{dist}_T(u,v)\leqslant\ell_0 }} \mathcal{L}_{u,v}^{\complement}. \]
Note that in order to show \eqref{eq:008}, it suffices to prove that for every $u,v\in V$ with $u\neq v$,
\begin{equation}\label{eq:028}
\mathbb{P}\big(\mathcal{A}_{u,v}^{\complement}\,\big|\,\mathcal{L}\big)\leqslant\frac{1}{N^2}.
\end{equation}
To that end, let $u,v\in V$ with $u\neq v$ and set $k:= \mathrm{dist}_T(u,v)$.
Notice that
\begin{equation}\label{eq:029}
  \boldsymbol{\zeta}(u)- \boldsymbol{\zeta}(v)
  = \Big( 1-\frac{1}{\log N} \Big) \Big(\sum_{e\in P(r_T, u)} \mathbf{Y}_e - \sum_{e\in P(r_T, v)} \mathbf{Y}_e\Big) +\delta \Big(\sum_{e\in P(r_T, u)} \mathbf{G}_e - \sum_{e\in P(r_T, v)} \mathbf{G}_e\Big);
\end{equation}
moreover, due to the independence of $(\mathbf{Y}_e)_{e\in E}$ and $(\mathbf{G}_e)_{e\in E}$, for every measurable
$A\subseteq \mathbb{R}^m$,
\begin{equation}\label{eq:030}
\mathbb{P}\Big(\sum_{e\in P(r_T, u)} \mathbf{G}_e - \sum_{e\in P(r_T, v)} \mathbf{G}_e \in A\; \Big|\, \mathcal{L}\Big)
  = \mathbb{P} \Big(\sum_{e\in P(u, v)} \mathbf{G}_e \in A\Big).
\end{equation}
We consider four cases according to the value of $k$.

\subsubsection*{Case 1: $k=1$}

In this case, we have $e:=\{u,v\}\in E$. Since $\mathbf{Y}_e$ is uniformly chosen from the unit ball of $X$,
by the definition of $\mathcal{A}_{u,v}$ and \eqref{eq:029} and \eqref{eq:030}, we have
\begin{equation} \label{eq:031}
\mathbb{P}\big(\mathcal{A}_{u,v}^{\complement}\, \big|\, \mathcal{L}\big) \leqslant
\mathbb{P}\big( \|\mathbf{G}_e\|_X> \delta^{-1} \log^{-1}N\big) \stackrel{\eqref{eq:026}}{\leqslant}
\mathbb{P}\big(\|\mathbf{G}_e\|_2^*> m^{-\frac{1}{2}}\delta^{-1} \log^{-1}N\big).
\end{equation}
Since $N_0$ can be taken sufficiently large, we see that
\begin{align*}
& \frac{\big( m^{-\frac{1}{2}}\delta^{-1}\log^{-1}N -\sqrt{m}\big)^2}{2} \\
& \ \ \ \ \ \ =
\frac{1}{2} \Big(\exp\Big( \Big(\frac{1}{2}-\frac{c_1}{2}\Big) \log^{c_2}N -\frac{1}{2} \log m - \log\log N\Big) -
\exp\Big(-\frac{1}{2}\log m\Big)\Big)^2 \\
& \ \ \ \ \ \
\geqslant \exp\big( (1-2c_1)\log^{c_2}N\big) > 2\log N+\log2.
\end{align*}
By \eqref{eq:027} and \eqref{eq:031}, we conclude that
\[ \mathbb{P}\big(\mathcal{A}_{u,v}^{\complement}\,\big|\,\mathcal{L}\big)
\leqslant 2\exp\big(-2\log N-\log2\big) = \frac{1}{N^2}. \]

\subsubsection*{Case 2: $k\in\{2,\dots,k_0\}$}

By the definitions of $\mathcal{L}$ and $\mathcal{A}_{u,v}$ and \eqref{eq:029}, \eqref{eq:030} and \eqref{eq:026},
\begin{align} \label{eq:032}
\mathbb{P}\big(\mathcal{A}_{u,v}^{\complement}\,\big|\,\mathcal{L}\big) &
\leqslant \mathbb{P}\Bigg( \delta\Big\|\sum_{e\in P(u, v)} \mathbf{G}_e\Big\|_X\geqslant
\Big(1-\frac{1}{\log N}\Big) \Big(1+\frac{1}{2\, m^{1-\kappa}}\Big)-1\Bigg) \\
& \leqslant \mathbb{P}\Bigg(\Big\|\frac{1}{\sqrt{k}}\sum_{e\in P(u, v)} \mathbf{G}_e\Big\|_2^*\geqslant
\frac{1}{\sqrt{k_0}\sqrt{m}\delta} \Big( \Big(1-\frac{1}{\log N}\Big) \Big(1+\frac{1}{2\, m^{1-\kappa}}\Big)-1\Big)\Bigg). \nonumber
\end{align}
Taking $N_0$ sufficiently large and invoking the definitions of $m$, $\ell_0,k_0$ and $\delta$, we see that
\begin{align*}
\frac12 & \Big(  \frac{1}{\sqrt{k_0}\sqrt{m}\delta} \Big( \Big(1-\frac{1}{\log N}\Big)
\Big(1+\frac{1}{2\, m^{1-\kappa}}\Big)-1\Big) - \sqrt{m}\Big)^2 \geqslant \frac12 \,
\Big(\frac{1}{\sqrt{k_0}\sqrt{m}\delta}\cdot\frac{1}{4\, m^{1-\kappa}}-\sqrt{m}\Big)^2 \\
& \geqslant \frac12 \Big( \exp\Big( \Big(\frac12-\frac{c_1}{2}\Big) \log^{c_2}N- \Big(\frac32-\kappa\Big) \log m +
\log\Big(\frac{1}{4\sqrt{k_0}}\Big)\Big)- \exp\Big(-\frac12\log m\Big) \Big)^2 \\
&\geqslant \exp\big( (1-2c_1) \log^{c_2}N\big) >2\log N+\log2.
\end{align*}
Hence, by \eqref{eq:027} and \eqref{eq:032}, we obtain that
\[ \mathbb{P}\big( \mathcal{A}_{u,v}^{\complement}\,\big|\,\mathcal{L}\big)
\leqslant 2\exp(-2\log N-\log2) = \frac{1}{N^2}. \]

\subsubsection*{Case 3: $k\in\{k_0+1,\dots, \ell_0\}$}

Using \eqref{eq:026}, the fact that $1-\frac{1}{\log N}\geqslant\frac{1}{2}$ and that in this case we have that
$k>k_0\geqslant 4^{\frac{1}{c_1}}$, we see that
\begin{align} \label{eq:033}
\mathbb{P}\big(\mathcal{A}_{u,v}^{\complement}\,\big|\,\mathcal{L}\big) & \leqslant
\mathbb{P}\Bigg(\delta\Big\|\sum_{e\in P(u, v)} \mathbf{G}_e\Big\|_X\geqslant
\left(1-\frac{1}{\log N}\right)k^{c_1}-1\Bigg) \\
& \leqslant \mathbb{P}\Bigg( \Big\|\frac{1}{\sqrt{k}}\sum_{e\in P(u, v)} \mathbf{G}_e\Big\|_X \geqslant
\frac{k^{c_1-\frac{1}{2}}}{4\delta}\Bigg) \leqslant
\mathbb{P}\Bigg( \Big\|\frac{1}{\sqrt{k}}\sum_{e\in P(u, v)} \mathbf{G}_e\Big\|_2^*\geqslant
\frac{\ell_0^{c_1-\frac{1}{2}}}{4\delta\sqrt{m}}\Bigg). \nonumber
\end{align}
On the other hand, by the definitions of $\ell_0$ and $\delta$,
\[  \frac{1}{\delta}\, \ell_0^{c_1-\frac12} =
\exp\Big(\Big(\frac12-\frac{c_1}{2}\Big)\log^{c_2}N +
\Big(c_1-\frac12\Big)\log^{c_2}N\Big)=\exp\Big(\frac{c_1}{2}\log^{c_2}N\Big), \]
and so, taking $N_0$ sufficiently large and using our starting assumption on the dimension $m$,
\begin{align*}
\frac12\, \Big( \frac{\ell_0^{c_1-\frac{1}{2}}}{4\delta\sqrt{m}} - \sqrt{m}\Big)^2
& = \frac12\, \Big( \exp\Big(\frac{c_1}{2}\log^{c_2}N- \frac12\log m +
\log\frac14\Big)- \exp\Big( -\frac12\log m\Big) \Big)^2 \\
&\geqslant \exp\Big( \frac{c_1}{2}\log^{c_2}N\Big) >2\log N+\log2.
\end{align*}
By \eqref{eq:027} and \eqref{eq:033}, we obtain that
\[ \mathbb{P}\big( \mathcal{A}_{u,v}^{\complement}\, \big|\, \mathcal{L}\big)
\leqslant 2\exp\big(-2\log N-\log2\big) = \frac{1}{N^2}. \]

\subsubsection*{Case 4: $k>\ell_0$}

By \eqref{eq:029}, \eqref{eq:030} and \eqref{eq:026} and the fact that $\mathbf{G}_e$ is symmetric for every $e\in E$, we have
\begin{align} \label{eq:034}
\mathbb{P}\big( \mathcal{A}_{u,v}^{\complement}\, \big|\, \mathcal{L}\big) & =
\mathbb{P}\big( \| \boldsymbol{\zeta}(u) - \boldsymbol{\zeta}(v)\|_X\leqslant 1\, \big|\,\mathcal{L}\big) \leqslant
\mathbb{P}\big( \| \boldsymbol{\zeta}(u) - \boldsymbol{\zeta}(v)\|_2^*\leqslant 1\, \big|\, \mathcal{L}\big) \\
& = \mathbb{P} \Big( \delta\sum_{e\in P(u,v)} \mathbf{G}_e \in -\Big(1-\frac{1}{\log N}\Big) \Big(\sum_{e\in P(r_T,u)} \mathbf{Y}_e - \sum_{e\in P(r_T,v)} \mathbf{Y}_e\Big)
+ B_{\|\cdot\|_2^*} \, \Big|\, \mathcal{L}\Big) \nonumber \\
& \leqslant \mathbb{P} \Big( \delta\sum_{e\in P(u,v)} \mathbf{G}_e \in  B_{\|\cdot\|_2^*}\, \Big|\,  \mathcal{L}\Big)
\leqslant \mathbb{P}\Bigg(\Big\|\frac{1}{\sqrt{k}}\sum_{e\in P(u,v)} \mathbf{G}_e\Big\|_2^*
\leqslant\frac{1}{\delta\sqrt{\ell_0}}\Bigg). \nonumber
\end{align}
Next observe that, by the choices of $\ell_0$ and $\delta$,
\begin{equation} \label{eq:035}
\frac{1}{\delta\sqrt{\ell_0}} = \exp\Big(-\frac{c_1}{2} \log^{c_2}N\Big).
\end{equation}
Moreover, the random variable $\big\|\frac{1}{\sqrt{k}}\sum_{e\in P(u,v)} \mathbf{G}_e\big\|_2^*$
has the same distribution with $\|\mathbf{G}\|_2$, where $\mathbf{G}$ denotes a random vector in $\mathbb{R}^m$
with i.i.d. standard normal entries. Thus, by \eqref{eq:034},
\begin{equation}\label{eq:036}
\mathbb{P}\big(\mathcal{A}_{u,v}^{\complement}\, \big|\, \mathcal{L}\big) \leqslant
\mathbb{P}\left( \|\mathbf{G}\|_2 \leqslant\frac{1}{\delta\sqrt{\ell_0}}\right).
\end{equation}
Since the density function
of $\mathbf{G}$ is pointwise bounded by $1$, this implies
\begin{equation} \label{eq:037}
\mathbb{P}\big( \mathcal{A}_{u,v}^{\complement}\,\big|\,\mathcal{L}\big)
\leqslant \Big(\frac{C_B}{\delta\sqrt{\ell_0}}\Big)^m,
\end{equation}
where $C_B$ is as in Lemma \ref{lem:ball_volume}. (Here, more straightforward and sharper estimates are, of course, available; we are using Lemma \ref{lem:ball_volume} for uniformity in the argument.) Finally, taking $N_0$ sufficiently large, we conclude that
\begin{align*}
\mathbb{P}\big( \mathcal{A}_{u,v}^{\complement}\,\big|\,\mathcal{L}\big)
& \leqslant \exp\Big( m\Big(-\frac{c_1}{2}\log^{c_2}N + \log C_B\Big)\Big)
\leqslant \exp\Big( -m\frac{c_1}{4}\log^{c_2}N \Big) \\
& \leqslant \exp\Big(-\frac{\alpha_1 c_1}{4}\cdot \frac{\log^{1+c_2}N}{\log\log N} \Big)
\leqslant \exp\big(-2\log N\big) = \frac{1}{N^2}.
\end{align*}

The above cases are exhaustive, and so the entire proof of Theorem \ref{thm:complete-tree-embedding} is completed.

\begin{remark}
By slightly modifying the argument, it is not hard to see that, with positive probability,
the random embedding $\boldsymbol{\zeta}$ constructed in the proof of part $(A)$ of
Theorem \ref{thm:complete-tree-embedding} satisfies
\begin{enumerate}
\item[$\bullet$] $\|\boldsymbol{\zeta}(u)- \boldsymbol{\zeta}(v)\|_X \leqslant 2 \dist_T(u,v)$
for all $u,v\in T$, and
\item[$\bullet$] $\|\boldsymbol{\zeta}(u)- \boldsymbol{\zeta}(v)\|_X \gtrsim \sqrt[10]{\dist_T(u,v)}$
for all $u,v\in T$ with $\dist_T(u,v)>k_0$,
\end{enumerate}
where $k_0$ is as \eqref{eq:009}.
\end{remark}

\subsection*{Acknowledgments}

We would like to thank the anonymous referees for their comments, remarks and suggestions that helped us improve the exposition.

The research was supported in the framework of H.F.R.I call ``Basic research Financing (Horizontal support of all Sciences)" under the National Recovery and Resilience Plan ``Greece 2.0" funded by the European Union--NextGenerationEU (H.F.R.I. Project Number: 15866). The third named author (K.T.) is partially supported by the NSF grant DMS 2331037.

\end{document}